\documentclass[11pt,english]{article}
\usepackage{amssymb}
\usepackage{amsfonts,yfonts}
\usepackage[all]{xy}
\usepackage{amsfonts,amscd,amssymb,amsthm}
\usepackage{amsfonts,amssymb,amsmath,amsxtra,amscd,amsthm,eucal,graphicx,graphics}
\usepackage{tikz}
\usepackage{enumerate}
 \usepackage[T1]{fontenc}
\usepackage[utf8]{inputenc}
\usepackage{babel}

\theoremstyle{definition}
\newtheorem{theorem}{Theorem}[section]

\newtheorem{proposition}[theorem]{Proposition}
\newtheorem{lemma}[theorem]{Lemma}

\newtheorem{definition}[theorem]{Definition}
\newtheorem{conjecture}[theorem]{Conjecture}

 \title{Maximizing  the number of $x$-colorings of $4$-chromatic graphs}
    \author{Aysel Erey\footnote{Department of Mathematics, University of Denver, aysel.erey@gmail.com}}
    \date{  \today }

\begin{document}

\maketitle

\begin{abstract}
	Let $\mathcal{C}_4(n)$ be the family of all connected $4$-chromatic graphs of order $n$. Given an integer $x\geq 4$, we consider the problem of finding the maximum number of $x$-colorings of a graph in $\mathcal{C}_4(n)$. It was conjectured that the maximum number of $x$-colorings is equal to $(x)_{\downarrow 4}(x-1)^{n-4}$ and the extremal graphs are those which have clique number $4$ and size $n+2$.
	
	In this article, we reduce this problem to a \textit{finite} family of graphs. We  show that there exist a finite family $\mathcal{F}$ of connected $4$-chromatic graphs such that if the number of $x$-colorings of every graph $G$ in $\mathcal{F}$ is less than $(x)_{\downarrow 4}(x-1)^{|V(G)|-4}$ then the conjecture holds to be true.
	
\end{abstract}

\noindent 

\section{Introduction}

In recent years problems of maximizing the number of colorings over various families of graphs have received a considerable amount of attention in the literature, see, for example, \cite{brownerey,engbers, engbersgalvin, dongbook, ereyjoc, loh, ma, tofts, zhao}. A natural graph family to look at is the family of connected graphs with fixed chromatic number and fixed order. Let $\mathcal{C}_k(n)$ be the family of all connected $k$-chromatic graphs of order $n$. What is the maximum number of $k$-colorings among all graphs in $\mathcal{C}_k(n)$? Or more generally, for an integer $x\geq k$, what is the maximum number of $x$-colorings of a graph in $\mathcal{C}_k(n)$ and what are the extremal graphs? The answer to this question depends on the chromatic number $k$. When $k\leq 3$, the answer to this question is known and when $k\geq 4$ the problem is wide open. It is well known that (see, for example, \cite{dongbook})  for $k=2$ and $x\geq 2$, the maximum number of $x$-colorings of a graph in $\mathcal{C}_2(n)$ is equal to $x(x-1)^{n-1}$, and extremal graphs are trees when $x\geq 3$. For $k=3$, Tomescu~\cite{tomescu} settled the problem by showing the following:

\begin{theorem}\cite{tomescu} \label{3chromtomes} If $G$ is a graph in  $\mathcal{C}_3(n)$ then
$$\pi(G,x)\leq (x-1)^{n}-(x-1) \ \ \ \text{for odd} \ n$$
and
$$\pi(G,x)\leq (x-1)^{n}-(x-1)^2 \ \ \ \text{for even} \ n$$
 for every integer $x\geq 3$. Furthermore, the extremal graph is the odd cycle $C_{n}$ when $n$ is odd and odd cycle with a vertex of degree $1$ attached to the cycle (denoted $C_{n-1}^1$) when $n$ is even.
 \end{theorem}

Let $\mathcal{C}^*_k(n)$ be the set of all graphs in $\mathcal{C}_k(n)$ which have clique number $k$ and size ${k \choose 2}+n-k$ (see Figure~\ref{extremalgraphspicture}). It is easy to see that if $G\in \mathcal{C}^*_k(n)$ then $\pi(G,x)=(x)_{\downarrow k}\,(x-1)^{n-k}$ where $(x)_{\downarrow k}$ is the $k$th falling factorial $x(x-1)(x-2)\cdots (x-k+1)$. Tomescu~\cite{tomescufrench} conjectured that when $k\geq 4$, the maximum number of $k$-colorings of a graph in $\mathcal{C}_k(n)$ is equal to $k!(k-1)^{n-k}$ and extremal graphs belong to $\mathcal{C}^*_k(n)$.

\begin{conjecture}\cite{tomescufrench}
If $G\in \mathcal{C}_k(n)$ where $k\geq 4$ then $$\pi(G,k)\leq k!\,(k-1)^{n-k}$$ and extremal graphs belong to $\mathcal{C}^*_k(n)$.
\end{conjecture}

The conjecture above was later extended to all $x$-colorings with $x\geq 4$.

 \begin{conjecture}\cite[pg. 315]{dongbook} \label{tomesdongconj} Let $G$ be a graph in  $\mathcal{C}_k(n)$  where $k\geq 4$. Then for every $x\in \mathbb{N}$ with $x\geq k$
 	$$\pi(G,x)\leq (x)_{\downarrow k}(x-1)^{n-k}.$$
 	Moreover, the equality holds if and only if $G$ belongs to  $\mathcal{C}^*_k(n)$.
 \end{conjecture}

Several authors have studied Conjecture~\ref{tomesdongconj}. In \cite{tomescu}, Conjecture~\ref{tomesdongconj} was proven for $k=4$ under the additional condition that graphs are planar:

\begin{theorem}\cite{tomescu}\label{tomesplanar}
	If $G$ is a planar graph in $\mathcal{C}_4(n)$ then
	$$\pi(G,x)\leq (x)_{\downarrow 4}(x-1)^{n-4}$$
	for every integer $x\geq 4$ and furthermore equality holds if and only if $G$ belongs to  $\mathcal{C}^*_4(n)$.
\end{theorem}

Also, in \cite{brownerey}  Conjecture~\ref{tomesdongconj}  was proven for every $k\geq 4$, provided that $x\geq n-2+\left( {n\choose 2}-{k\choose 2}-n+k \right) ^2$, and in \cite{ereyjoc} it was proven for every $k\geq 4$ under the additional condition that independence number of the graphs is at most $2$. In this article, our main result is Theorem~\ref{main} which reduces this conjecture (for $k=4$) to a \textit{finite} family of $4$-chromatic graphs.

\begin{figure}[h]\label{extremalgraphspicture}
	\begin{center}
			\includegraphics[width=0.8\textwidth]{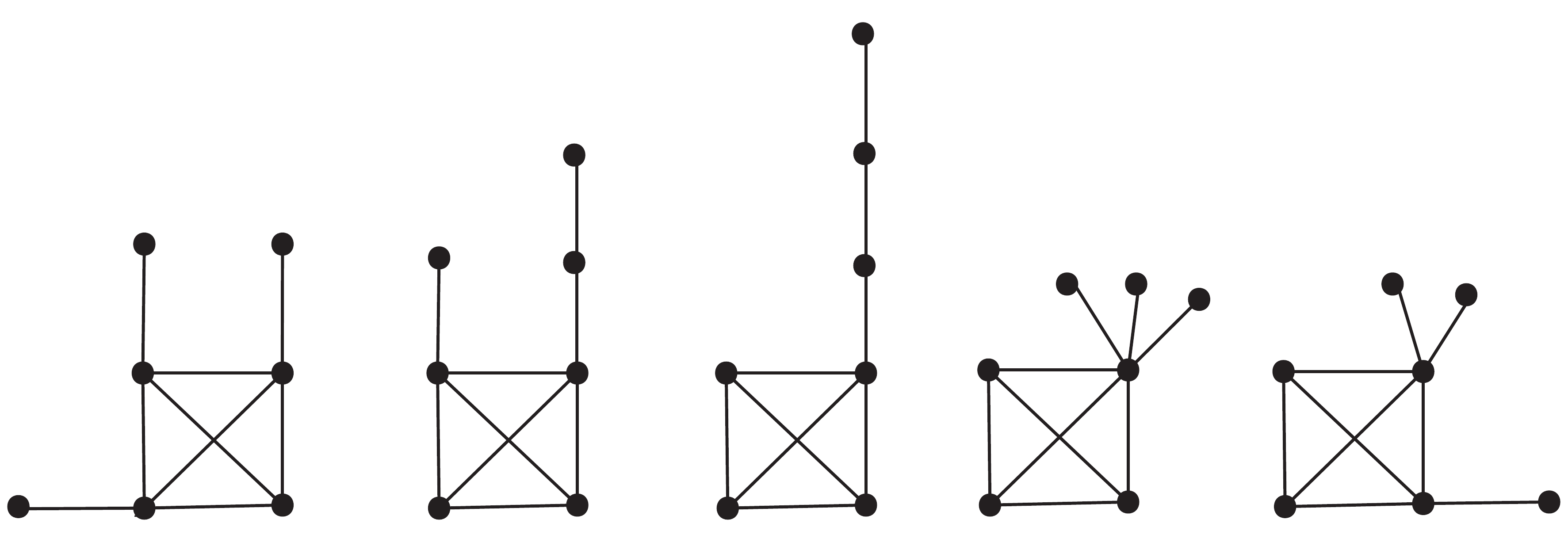}
	\end{center}
	\caption{The graphs in the family $\mathcal{C}^*_4(7)$.}
\end{figure}

\section{Terminology and background}

Let  $V(G)$ and $E(G)$ be the vertex set and edge set of a (finite, undirected) graph $G$, respectively. The {\em order} of $G$ is $|V(G)|$  and the {\em size} of $G$ is $|E(G)|$. For a nonnegative integer $x$, a (proper) \textit{$x$-coloring} of $G$ is a function $f:V(G)\rightarrow \{1,\dots , x\}$ such that $f(u)\neq f(v)$ for every $uv\in E(G)$. The \textit{chromatic number} $\chi(G)$ is smallest $x$ for which $G$ has an $x$-coloring and $G$ is called \textit{k-chromatic} if $\chi(G)=k$. Let $\pi(G,x)$ denote the {\em chromatic polynomial of G}. For nonnegative integers $x$, the polynomial $\pi(G,x)$ counts the number of $x$-colorings of $G$.

Let $G+e$ be the graph obtained from $G$ by adding an edge $e$ and $G / e$ be the graph formed from $G$ by {\it contracting} edge $e$. For $e\notin E(G)$, observe that 
\begin{center}
	$\chi(G)=\operatorname{min}\{\chi(G+e)\, , \, \chi(G/ e)\}$
\end{center}
and
\begin{equation}\label{contaddremark}
	|\chi(G+e)-\chi(G/e)|\leq 1.
\end{equation}

The well known  \textit{Addition-Contraction Formula} says that   $$\pi(G,x)=\pi(G+e,x)+\pi(G/ e,x).$$

A graph $G$ is called the $r$-\textit{clique sum} of $G_1,G_2,\dots ,G_n$ if $G=G_1\cup G_2\cup \cdots G_n$ and $G_1\cap G_2\cap \cdots G_n$ induces a complete graph $K_r$ in $G$ (see Figure~\ref{cliquesumpicture}). In this case the {\em Complete Cutset Theorem} says that
$$\pi(G,x)=\frac{\prod\limits_{i=1}^n\, \pi(G_i,x)}{\left((x)_{\downarrow r}\right)^{n-1}}.$$ 

\begin{figure}[h]\label{cliquesumpicture}
	\begin{center}
			\includegraphics[width=0.25\textwidth]{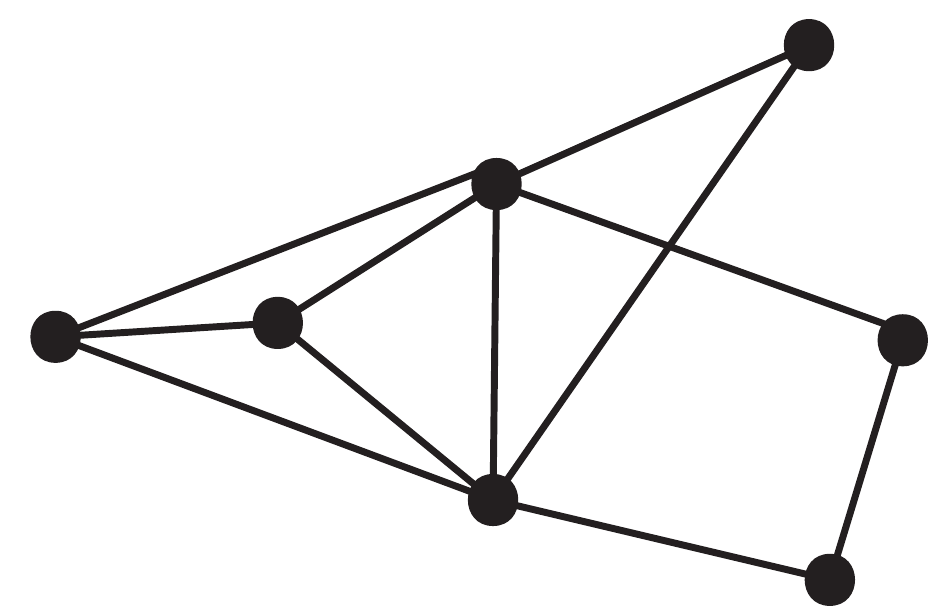}
	\end{center}
	\caption{The $2$-clique sum of $C_3$, $C_4$, $K_4$.}
\end{figure}

A subset $S$ of the vertices of a graph $G$ is called a \textit{cutset} of $G$ if $G-S$ has more than one component. A connected graph is called $k$-\textit{connected} if there does not exist a set of $k-1$ vertices whose removal disconnects the graph. A \textit{block} of a graph $G$ is a maximal $2$-connected subgraph of $G$. 
A connected graph $G$ is called a \textit{cactus graph} if every block of $G$ is either an edge or a cycle. If $B_1,\dots ,B_n$ be the blocks of a connected graph $G$ then by the Complete Cutset Theorem, 
\begin{equation}\label{blocks}
\pi(G,x)=\frac{1}{x^{n-1}}\prod\limits_{i=1}^{n} \pi(B_i,x)
\end{equation}

The chromatic polynomial of a cycle graph $C_n$ is given by
$$\pi(C_n,x)=(x-1)^n+(-1)^n(x-1).$$

A graph $G'$ is called a \textit{subdivision} of $G$ if $G'$ is obtained from $G$ by replacing edges of $G$ with paths whose endpoints are the vertices of the edges. Let $K_{p,q}$ denote the complete bipartite graph with partitions of size $p$ and $q$.  The $t$-\textit{spoke wheel}, denoted by $W_t$, has vertices $v_0,v_1,\dots ,v_t$ where $v_1,v_2,\dots ,v_t$ form a cycle, and $v_0$ is adjacent to all of $v_1,v_2,\dots ,v_t$. Let $V_t$ denote the graph whose vertex set is $\{u_1,u_2,\dots ,u_t, v_2,\dots ,v_{t-1}\}$ and edge set is 
 $$\{u_iu_{i+1}\}_{i=1}^{t-1}\cup \, \{v_iv_{i+1}\}_{i=2}^{t-2} \cup \, \{u_iv_{i}\}_{i=2}^{t-1} \cup  \{u_1v_2,\, u_tv_{t-1},\, u_1u_t\}$$
see Figure~\ref{laddertypepicture}.

\begin{proposition}\label{subgraphpropn}
	If $H$ is a connected subgraph of a connected graph $G$, then for all $x\in \mathbb{N}$, $$\pi(G,x)\leq \pi(H,x)(x-1)^{|V(G)|-|V(H)|}.$$ 
\end{proposition}

\begin{proof}
	Let $G'$ be a minimal connected spanning subgraph of $G$ which contains $H$. Then, by the Complete Cutset Theorem, $\pi(G',x)=\pi(H,x)(x-1)^{|V(G)|-|V(H)|}$. Every $x$-coloring of $G$ is an $x$-coloring of $G'$. Hence, $\pi(G',x)\geq \pi(G,x)$. Thus the result follows.
\end{proof}

\begin{proposition}\label{general_k_clique_lemma}\cite{ereyjoc}
	Let $G\in \mathcal{C}_k(n)$ and $\omega(G)=k$. Then for all $x\in \mathbb{N}$ with $x\geq k$, $$\pi(G,x)\leq (x)_{\downarrow k}\, (x-1)^{n-k}$$ with equality if and only if $G\in  \mathcal{C}_k^*(n) $.
\end{proposition}

\section{Proof of the main result}

To prove our main result, we need the following three lemmas whose proofs are provided in Section $4$.

\begin{lemma}\label{3connreduction} Let $x\in \mathbb{N}$ be such that $x\geq 4$. Suppose that for every noncomplete $3$-connected $4$-chromatic graph $H$, the inequality $\pi(H,x)<(x)_{\downarrow 4}(x-1)^{|V(H)|-4}$ holds. Then, for every connected $4$-chromatic graph $G$ the inequality $\pi(G,x)\leq (x)_{\downarrow 4}(x-1)^{|V(G)|-4}$ holds with equality if and only if $G\in \mathcal{C}_4^*(|V(G)|)$.
\end{lemma}

\begin{lemma}\label{K33son}
	Let $G$ be a subdivision of $K_{3,10}$ and $|V(G)|=n$. Then, 
	$$\pi(G,x)<(x)_{\downarrow 4}\,(x-1)^{n-4}$$
	for every real number $x\geq 3.95$.
\end{lemma}

\begin{lemma}\label{cactusson} Let $G$ be a cactus graph of order $n$ which has $6$ cycles. Then,
	$$\pi(G,x)<(x)_{\downarrow 4}(x-1)^{n-4}$$
	for every real number $x\geq 3.998$.
\end{lemma}

We also make use of the following result.

\begin{theorem}\cite{oxley} \label{oxley}
For every integer $t\geq 3$, there is an integer $N=f(t)$ such that every $3$-connected graph with at least $N$ vertices contains a subgraph isomorphic to a subdivision of one of $W_t$, $V_t$, and $K_{3,t}$.
\end{theorem}

\begin{figure}[h]\label{laddertypepicture}
	\begin{center}
			\includegraphics[width=0.5\textwidth]{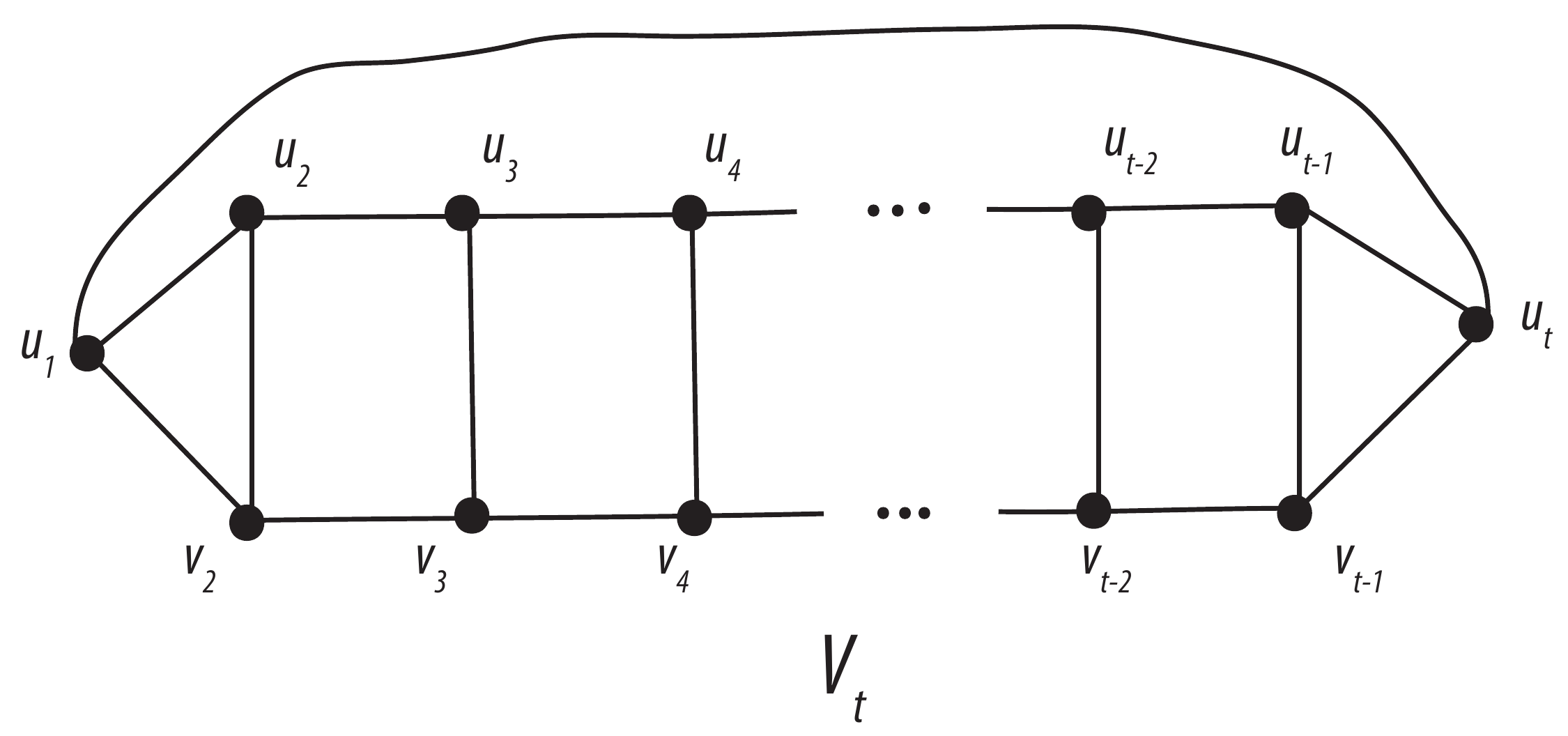}
	\end{center}
	\caption{The graph $V_t$ in Theorem~\ref{oxley}.}
\end{figure}

\begin{theorem}\label{main}
There exists a finite family $\mathcal{F}$ of $3$-connected nonplanar $4$-chromatic graphs such that if every graph $G$ in $\mathcal{F}$ satisfies $\pi(G,x)<(x)_{\downarrow 4}(x-1)^{|V(G)|-4}$ for all $x\in \mathbb{N}$ with $x\geq 4$, then Conjecture~\ref{tomesdongconj} holds to be true.
\end{theorem}
\begin{proof}
Take $t=12$ in Theorem~\ref{oxley} and let $N=f(12)$. Let $\mathcal{F}$ be the family of all $3$-connected nonplanar $4$-chromatic graphs of order less than $N$. Assume that for every graph $G$ in $\mathcal{F}$, the inequality $\pi(G,x)<(x)_{\downarrow 4}(x-1)^{|V(G)|-4}$ holds for every integer $x\geq 4$. Now we shall show that Conjecture~\ref{tomesdongconj} holds to be true. Let $x\in \mathbb{N}$ with $x\geq 4$. By Lemma~\ref{3connreduction} it suffices to show that every noncomplete $3$-connected $4$-chromatic graph $H$ satisfies $\pi(H,x)<(x)_{\downarrow 4}(x-1)^{|V(H)|-4}$. Let $H$ be a $3$-connected $4$-chromatic graph. By Theorem~\ref{tomesplanar}, we may assume that $H$ is nonplanar. If $|V(H)|<N$ then the result holds by the assumption. So we may assume that $|V(H)|\geq N$. By Theorem~\ref{oxley}, $H$ contains a subgraph isomorphic to a subdivision of $W_{12}$, $V_{12}$ and $K_{3,12}$. If $H$ contains a subgraph isomorphic to a subdivision of $K_{3,12}$ then the result follows by Proposition~\ref{subgraphpropn} and Lemma~\ref{K33son}. If $H$ contains a subgraph isomorphic to a subdivision of  $W_{12}$ or $V_{12}$ then $H$ contains a subgraph isomorphic to cactus graph having $6$ cycles. Therefore the result follows from Proposition~\ref{subgraphpropn} and Lemma~\ref{cactusson}.
\end{proof}

\section{Proofs of lemmas used in the proof of the main result}

\subsection{Reduction to $3$-connected graphs}

Let $S$ be a set of vertices in a graph $G$. An $S$-\textit{lobe} of $G$ is an induced subgraph of $G$ whose vertex set consists of $S$ and the vertices of a component of $G-S$. A $k$-chromatic graph $G$ is called $k$-\textit{critical} if $\chi(H)<\chi(G)$ for every proper subgraph $H$ of $G$.

\begin{proposition}\cite[pg. 218]{westbook}\label{westlobe} Let $G$ be a $k$-critical graph with a cutset $S=\{x,y\}$. Then 
\begin{itemize}
\item[(i)] $xy\notin E(G)$, and 
\item[(ii)] $G$ has exactly two $S$-lobes and they can be named $G_1$, $G_2$ such that $G_1+xy$ is $k$-critical and $G_2/xy$ is $k$-critical.
\end{itemize}
\end{proposition}

\noindent \textbf{Proof of Lemma~\ref{3connreduction}.} We proceed by induction on the number of edges. If $G\in \mathcal{C}_4^*(|V(G)|)$, then the equality $\pi(G,x)= (x)_{\downarrow 4}(x-1)^{|V(G)|-4}$ holds and the result is clear. The minimum number of edges of a connected $4$-chromatic graph $G$ which does not belong to $\mathcal{C}_4^*(|V(G)|)$ is $8$ and the extremal graph is the union of a $K_4$ and $K_3$ which intersect in an edge. So $\pi(G,x)=\frac{(x)_{\downarrow 4}(x)_{\downarrow 3}}{(x)_{\downarrow 2}}=(x)_{\downarrow 4}(x-2)$ and the strict inequality $(x)_{\downarrow 4}(x-2)<(x)_{\downarrow 4}(x-1)$ holds. Now suppose that $G$ is a connected $4$-chromatic graph with $|E(G)|>8$ and $G\notin \mathcal{C}_4^*(|V(G)|)$.

If $G$ is not $2$-connected, then $G$ has a block $B$ such that $|E(B)|<|E(G)|$ and $\chi(B)=4$ as $\chi(G)=\operatorname{max}\{\chi(B):\, B\text{ is a block of} \ G\}$. If $B\cong K_4$ then the result follows by Proposition~\ref{general_k_clique_lemma}. Suppose $B\ncong K_4$, then $B\notin \mathcal{C}_4^*(|V(B)|)$ as $B$ is $2$-connected and the only $2$-connected graph in $\mathcal{C}_4^*(|V(B)|)$ is the complete graph. By the induction hypothesis we have $\pi(B,x)<(x)_{\downarrow 4}(x-1)^{|V(B)|-4}$. By Proposition~\ref{subgraphpropn}, we have $\pi(G,x)\leq \pi(B,x)(x-1)^{|V(G)|-|V(B)|}$. Hence we get
$\pi(G,x)<(x)_{\downarrow 4}(x-1)^{|V(G)|-4}.$

Now we may assume that $G$ is $2$-connected. If $G$ is not $4$-critical then there is an edge $e\in E(G)$ such that $\chi(G-e)=4$. Also $G-e$ is connected as $G$ is $2$-connected. If $G-e$ is not $2$-connected then we can repeat the same argument as in the previous case to show that $\pi(G-e,x)\leq (x)_{\downarrow 4}(x-1)^{|V(G-e)|-4}$ with equality if and only if $G-e\in \mathcal{C}_4^*(|V(G-e)|)$. Note that $V(G)=V(G-e)$. If $G-e\in \mathcal{C}_4^*(|V(G)|)$ then $\chi(G/e)\geq 4$ and hence $\pi(G/e,x)>0$. If $G-e\notin \mathcal{C}_4^*(|V(G)|)$ then $\pi(G-e,x)<(x)_{\downarrow 4}(x-1)^{|V(G)|-4}$ by the induction hypothesis. In each case we get
$$\pi(G,x)=\pi(G-e,x)-\pi(G/e,x)<(x)_{\downarrow 4}(x-1)^{|V(G)|-4}.$$

For the rest of the proof we may assume that $G$ is a $4$-critical graph and $G$ is not $3$-connected. Let $S=\{u,v\}$ be a cutset of $G$. By Proposition~\ref{westlobe}, $uv\notin E(G)$ and $G$ has exactly two $S$-lobes and they can be named as $G_1$, $G_2$ such that $G_1+uv$ is $4$-critical and $G_2/uv$ is $4$-critical. So by the induction hypothesis, we have 
$$\pi(G_1+uv,x)\leq (x)_{\downarrow 4}\,(x-1)^{|V(G_1+uv)|-4}$$ and
$$\pi(G_2/uv,x)\leq (x)_{\downarrow 4}\,(x-1)^{|V(G_2/uv)|-4}.$$ By the observation in \eqref{contaddremark}, the inequalities $3\leq \chi(G_2+uv)\leq 5$ and $3\leq \chi(G_1/uv)\leq 5$ hold. If $\chi(G_2+uv)=3$ then by Theorem~\ref{3chromtomes}, 
$$\pi(G_2+uv,x)\leq (x-1)^{|V(G_2+uv)|}-(x-1).$$
If $\chi(G_2+uv)\geq 4$ then let $G'$ be a $4$-chromatic connected  spanning subgraph of $G_2+uv$. By the induction hypothesis, 
$$\pi(G',x)\leq (x)_{\downarrow 4}\,(x-1)^{|V(G')|-4}=(x)_{\downarrow 4}\,(x-1)^{|V(G_2+uv)|-4}.$$

Since $\pi(G_2+uv,x)\leq \pi(G',x)$, we get $\pi(G_2+uv,x)\leq (x)_{\downarrow 4}\,(x-1)^{|V(G_2+uv)|-4}.$ Now it is easy to check that 
$$(x)_{\downarrow 4}\,(x-1)^{|V(G_2+uv)|-4}\leq (x-1)^{|V(G_2+uv)|}-(x-1).$$
Hence, in each case we have 
$$\pi(G_2+uv,x)\leq (x-1)^{|V(G_2+uv)|}-(x-1).$$
 Similarly, we also have 
$$\pi(G_1/uv,x)\leq (x-1)^{|V(G_1/uv)|}-(x-1).$$
By the Complete Cutset Theorem,
\begin{eqnarray*}
\pi(G+uv,x) &=& \frac{\pi(G_1+uv,x)\,\pi(G_2+uv,x)}{x(x-1)}\\
&\leq & \frac{(x)_{\downarrow 4}\,(x-1)^{|V(G_1+uv)|-4}\,\left((x-1)^{|V(G_2+uv)|}-(x-1)\right)}{x(x-1)}\\
&=& \frac{(x)_{\downarrow 4}\left((x-1)^{|V(G)|-3}-(x-1)^{|V(G_1)|-4}\right)}{x}
\end{eqnarray*}
where the last equality follows since $|V(G_1+uv)|=|V(G_1)|$, $|V(G_2+uv)|=|V(G_2)|$ and $|V(G)|=|V(G_1)|+|V(G_2)|-2$.
Similarly,
\begin{eqnarray*}
\pi(G/uv,x) &=& \frac{\pi(G_1/uv,x)\,\pi(G_2/uv,x)}{x}\\
&\leq & \frac{\left((x-1)^{|V(G_1/uv)|}-(x-1)\right)\,(x)_{\downarrow 4}\,(x-1)^{|V(G_2/uv)|-4}}{x}\\
&=& \frac{(x)_{\downarrow 4}\left((x-1)^{|V(G)|-4}-(x-1)^{|V(G_2)|-4}\right)}{x}
\end{eqnarray*}
as $|V(G_1/uv)|=|V(G_1)|-1$, $|V(G_2/uv)|=|V(G_2)|-1$. Now, let $|V(G)|=n$, $|V(G_1)|=n_1$ and $|V(G_2)|=n_2$. Then,
\begin{eqnarray*}
\pi(G,x)&=&\pi(G+uv,x)+\pi(G/uv,x)\\
&\leq & \frac{(x)_{\downarrow 4}\left((x-1)^{n-3}-(x-1)^{n_1-4}\right)}{x} \,+\, \frac{(x)_{\downarrow 4}\left((x-1)^{n-4}-(x-1)^{n_2-4}\right)}{x}\\
&=&\frac{(x)_{\downarrow 4}\left((x-1)^{n-3}+(x-1)^{n-4}-(x-1)^{n_1-4}-(x-1)^{n_2-4}\right)}{x}\\
&=&\frac{(x)_{\downarrow 4}\left(x\,(x-1)^{n-4}-(x-1)^{n_1-4}-(x-1)^{n_2-4}\right)}{x}\\
&=&(x)_{\downarrow 4}\left((x-1)^{n-4}-\frac{(x-1)^{n_1-4}}{x}-\frac{(x-1)^{n_2-4}}{x}\right)\\
&<&(x)_{\downarrow 4}\, (x-1)^{n-4}.
\end{eqnarray*}
Thus the result follows.
\qed

\subsection{Proof of Lemma~\ref{K33son}}

The chromatic polynomial of a subdivision of $K_{3,t}$ can be calculated using the chromatic polynomials of theta graphs and a  certain subdivision of $K_4$. So, in order to prove Lemma~\ref{K33son}, we shall first analyze theta graphs and a subdivision of $K_4$.

\subsubsection{Theta graphs}

A \textit{theta} graph $\theta_{s_1,s_2,s_3}$ is formed by taking a pair of vertices $u$, $v$ and joining them by three internally disjoint paths of sizes $s_1,s_2,s_3$ (see Figure~\ref{thetapicture}). By the Addition-Contraction Formula, it is easy to see that 
\begin{equation}\label{thetageneralformula}
\pi(\theta_{s_1,s_2,s_3},x)=\frac{\prod\limits_{i=1}^3\left((x-1)^{s_i+1}+(-1)^{s_i+1}(x-1)\right)}{\left(x(x-1)\right)^2}+\frac{\prod\limits_{i=1}^3\left((x-1)^{s_i}+(-1)^{s_i}(x-1)\right)}{x^2}
\end{equation} 
(see, for example, \cite{brownsokal} for details).

\begin{figure}[h]
	\begin{center}
			\includegraphics[width=0.2\textwidth]{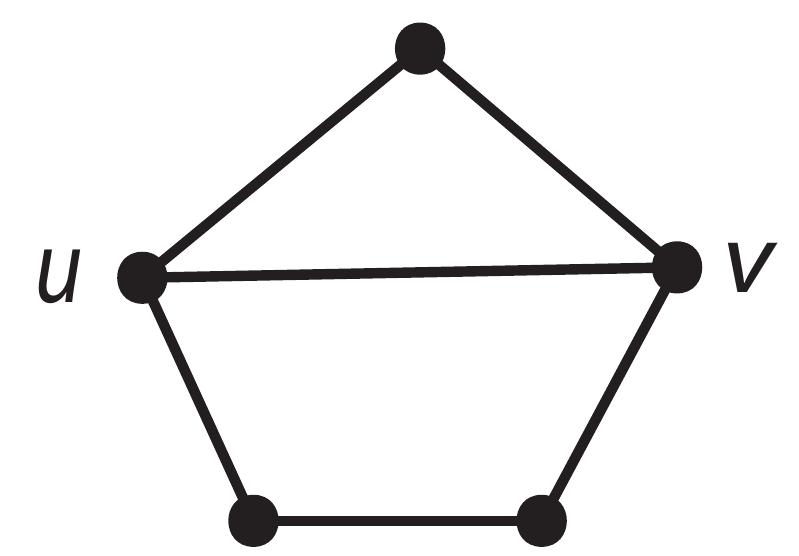}
	\end{center}
	\caption{The graph $\theta_{2,1,3}$.}
	\label{thetapicture}
\end{figure}

\begin{lemma}\label{thetashift}
$\pi(\theta_{s_1,s_2,s_3},x+1)$ is equal to $$\frac{x}{x+1}\left(x^{\left(\sum\limits_{i=1}^3s_i\right)-1}+(-1)^{s_1+s_2}x^{s_3}+(-1)^{s_1+s_3}x^{s_2}+(-1)^{s_2+s_3}x^{s_1}+(-1)^{\sum\limits_{i=1}^3s_i}(x-1)\right).$$
\end{lemma}
\begin{proof}
Using the formula given in \eqref{thetageneralformula},
\begin{eqnarray*}
\pi(\theta_{s_1,s_2,s_3},x+1) &=& \frac{\prod\limits_{i=1}^3 \left(x^{s_i+1}+(-1)^{s_i+1}x\right)}{x^2(x+1)^2}
+\frac{\prod\limits_{i=1}^3 \left(x^{s_i}+(-1)^{s_i}x\right)}{(x+1)^2}\\
& = & \frac{x}{(x+1)^2} \left(\prod\limits_{i=1}^3 \left(x^{s_i}+(-1)^{s_i+1}\right)
\, +\, x^2 \, \prod\limits_{i=1}^3 \left(x^{s_i-1}+(-1)^{s_i}\right) \right).\\
\end{eqnarray*}

Calculations show that the latter is equal to \\

$\frac{x}{(x+1)^2}(x^{s_1+s_2+s_3}+x^{s_1+s_2+s_3-1}+(-1)^{s_2+s_3}x^{s_1+1}+(-1)^{s_2+s_3}x^{s_1}+(-1)^{s_1+s_3}x^{s_2+1}+(-1)^{s_1+s_3}x^{s_2}$ 

$+(-1)^{s_1+s_2}x^{s_3+1}+(-1)^{s_1+s_2}x^{s_3}+(-1)^{s_1+s_2+s_3}x^2-(-1)^{s_1+s_2+s_3}).$\\

Now we rewrite the latter as \\

$\frac{x}{(x+1)^2}(x^{s_1+s_2+s_3-1}(x+1)+(-1)^{s_2+s_3}x^{s_1}(x+1)+(-1)^{s_1+s_3}x^{s_2}(x+1)$ 

$+(-1)^{s_1+s_2}x^{s_3}(x+1)+(-1)^{s_1+s_2+s_3}(x^2-1))$\\

which simplifies to\\

$\frac{x}{x+1}\left(x^{s_1+s_2+s_3-1}+(-1)^{s_1+s_2}x^{s_3}+(-1)^{s_1+s_3}x^{s_2}+(-1)^{s_2+s_3}x^{s_1}+(-1)^{s_1+s_2+s_3}(x-1)\right).$
\end{proof}

\begin{definition}
Given $a,b,c\in \mathbb{Z^{+}}$, we define a function $G_{a,b,c}$  by
\[
G_{a,b,c}(x) = 
\begin{cases}
1+\frac{3}{x^3}+\frac{1}{x^4} & \text{if}\ \text{none of } a,b,c \ \text{is equal to} \ 1 \\
1+\frac{2}{x^3}+\frac{1}{x^6} & \text{if}\ \text{exactly one of } a,b,c \ \text{is equal to} \ 1 \\
1+\frac{1}{x}+\frac{1}{x^3}+\frac{1}{x^4} & \text{if}\ \text{exactly two of } a,b,c \ \text{are equal to} \ 1 \\
1+\frac{2}{x}+\frac{1}{x^2} & \text{if}\ \text{all of } a,b,c \ \text{are equal to} \ 1 \\
\end{cases}
\]
\end{definition}
\begin{lemma}\label{thetabound}
	Let $a,b,c \in \mathbb{Z}^{+}$. Then for every real number $x\geq 1$,
	$$\pi(\theta_{a,b,c},x+1)\leq \frac{x^{a+b+c}}{x+1}\, G_{a,b,c}(x)$$
\end{lemma}
\begin{proof}
	By Lemma~\ref{thetashift}, $\pi(\theta_{a,b,c},x+1)$ is equal to  $$\frac{x}{x+1}\left(x^{a+b+c-1}+(-1)^{a+b}x^c+(-1)^{a+c}x^b+(-1)^{b+c}x^a+(-1)^{a+b+c}(x-1)\right).$$
	So, it suffices to show that 
	\begin{equation}\label{thetaboundproofineq}
x^{a+b+c-1}+(-1)^{a+b}x^c+(-1)^{a+c}x^b+(-1)^{b+c}x^a+(-1)^{a+b+c}(x-1) \, \leq \, G_{a,b,c}(x)\,x^{a+b+c-1}.
	\end{equation}

	To prove the inequality in \eqref{thetaboundproofineq}, we consider several cases.\\
	
	\noindent \underline{Case $1$}: $a,b,c \geq 2$. 
	
	By the definition of $G_{a,b,c}$, 
	$$G_{a,b,c}(x) \, x^{a+b+c-1}=x^{a+b+c-1}+3x^{a+b+c-4}+x^{a+b+c-5}.$$
	Each of $(-1)^{a+b}x^c$, $(-1)^{a+c}x^b$ and  $(-1)^{b+c}x^a$ is at most $x^{a+b+c-4}$. So, $$(-1)^{a+b}x^c+(-1)^{a+c}x^b+(-1)^{b+c}x^a\leq 3x^{a+b+c-4}.$$
	Also, it is clear that $(-1)^{a+b+c}(x-1)\leq x^{a+b+c-5}$. Now the inequality in \eqref{thetaboundproofineq} follows.\\

		\noindent \underline{Case $2$}: exactly one of $a,b$ and $c$ is equal to $1$. 
		
		Without loss, we may assume that $a=1$ and $b,c\geq 2$. By the definition of $G_{a,b,c}$, 
		$$G_{a,b,c}(x) \, x^{a+b+c-1}=x^{b+c}+2x^{b+c-3}+x^{b+c-6}.$$
		Also, the left side of \eqref{thetaboundproofineq} is equal to $$x^{b+c}+(-1)^{1+b}x^c+(-1)^{1+c}x^b+(-1)^{b+c}.$$
		
		If $b=c=2$ then	$G_{a,b,c}(x)x^{a+b+c-1}$ is equal to  $x^4+2x+x^{-2}$ and the left side of \eqref{thetaboundproofineq} is equal to  $x^{4}-2x^{2}+1$. And it is clear that $x^{4}-2x^{2}+1\leq x^4+2x+x^{-2}$.
		
If exactly one of $b$ and $c$ is equal to $2$, say, $b=2$ and $c\geq 3$, then	$G_{a,b,c}(x)x^{a+b+c-1}$ is equal to $x^{c+2}+2x^{c-1}+x^{c-4}$  and the left side of \eqref{thetaboundproofineq} is equal to $x^{c+2}-x^c+(-1)^{c+1}x^2+(-1)^c$. Now it is easy to see that $x^{c+2}-x^c+(-1)^{c+1}x^2+(-1)^c\leq x^{c+2}+2x^{c-1}+x^{c-4}$ since $c\geq 3$.

If $b,c\geq 3$ then each of $(-1)^{1+b}x^c$ and  $(-1)^{1+c}x^b$ is at most $x^{b+c-3}$. So, $(-1)^{1+b}x^c+(-1)^{1+c}x^b \leq 2x^{b+c-3}$. Also, $(-1)^{b+c}\leq x^{b+c-6}$. Therefore,
$$x^{b+c}+(-1)^{1+b}x^c+(-1)^{1+c}x^b+(-1)^{b+c}\leq x^{b+c}+2x^{b+c-3}+x^{b+c-6}.$$

\noindent \underline{Case $3$}: exactly two of $a,b$ and $c$ is equal to $1$. 
		
		Without loss, we may assume that $a=b=1$ and $c\geq 2$. By the definition of $G_{a,b,c}$, 
		$$G_{a,b,c}(x) \, x^{a+b+c-1}=x^{1+c}+x^{c}+x^{c-2}+x^{c-3}.$$
		Also, the left side of \eqref{thetaboundproofineq} is equal to $$x^{1+c}+x^c+(-1)^{1+c}(x+1).$$
		It is easy to see that $(-1)^{1+c}(x+1)\leq x^{c-2}+x^{c-3}$ since $c\geq 2$. So,
		$$x^{1+c}+x^c+(-1)^{1+c}(x+1)\leq x^{1+c}+x^{c}+x^{c-2}+x^{c-3}.$$
		
		\noindent \underline{Case $4$}: $a=b=c=1$.
		
	 By the definition of $G_{a,b,c}$, 
		$$G_{a,b,c}(x) \, x^{a+b+c-1}=x^2+2x+1.$$
		The left side of \eqref{thetaboundproofineq} is also  equal to $x^2+2x+1$. Therefore the result follows.
	\end{proof}
	
	\begin{lemma}\label{thetalastbound}
	Let $a,b,c\in \mathbb{Z^+}$ be such that at least one of $a$, $b$, $c$ is at least $2$. Then,
	$$\pi(\theta_{a,b,c},x+1)\leq \frac{x^{a+b+c}}{x+1}\left(1+\frac{1}{x}+\frac{1}{x^3}+\frac{1}{x^4}\right)$$
	for every real number $x\geq \sqrt{2}$.
	\end{lemma}
	
	\begin{proof} It is straightforward to check that 
	$$\frac{3}{x^3}+\frac{1}{x^4}\leq \frac{1}{x}+\frac{1}{x^3}+\frac{1}{x^4}$$
	and 
	$$\frac{2}{x^3}+\frac{1}{x^6}\leq \frac{1}{x}+\frac{1}{x^3}+\frac{1}{x^4}$$
	for all real $x\geq \sqrt{2}$. Thus the result follows by Lemma~\ref{thetabound}.
	\end{proof}

\subsubsection{A subdivision of $K_4$ }

Let $SK_4^{s_1,s_2,s_3}$ denote a subdivision of $K_4$ such that three edges of $K_4$ are replaced with paths of sizes $s_1$, $s_2$ and $s_3$, and all the other edges of $K_4$ are left undivided (see Figure~\ref{SK4picture}). If  $uv$ is an undivided edge of $K_4$, then $$SK_4^{s_1,s_2,s_3}-uv\cong \theta_{s_1+1,s_2,s_3+1}$$ and $$SK_4^{s_1,s_2,s_3}/uv \cong \theta_{s_1,s_2+1,s_3}.$$ Therefore,
\begin{equation}\label{SK4formula}
\pi (SK_4^{s_1,s_2,s_3},\, x)= \pi(\theta_{s_1+1,s_2,s_3+1}, \, x)-\pi(\theta_{s_1,s_2+1,s_3}, \, x)
\end{equation}\\

\begin{figure}[h]
	\begin{center}
			\includegraphics[width=0.3\textwidth]{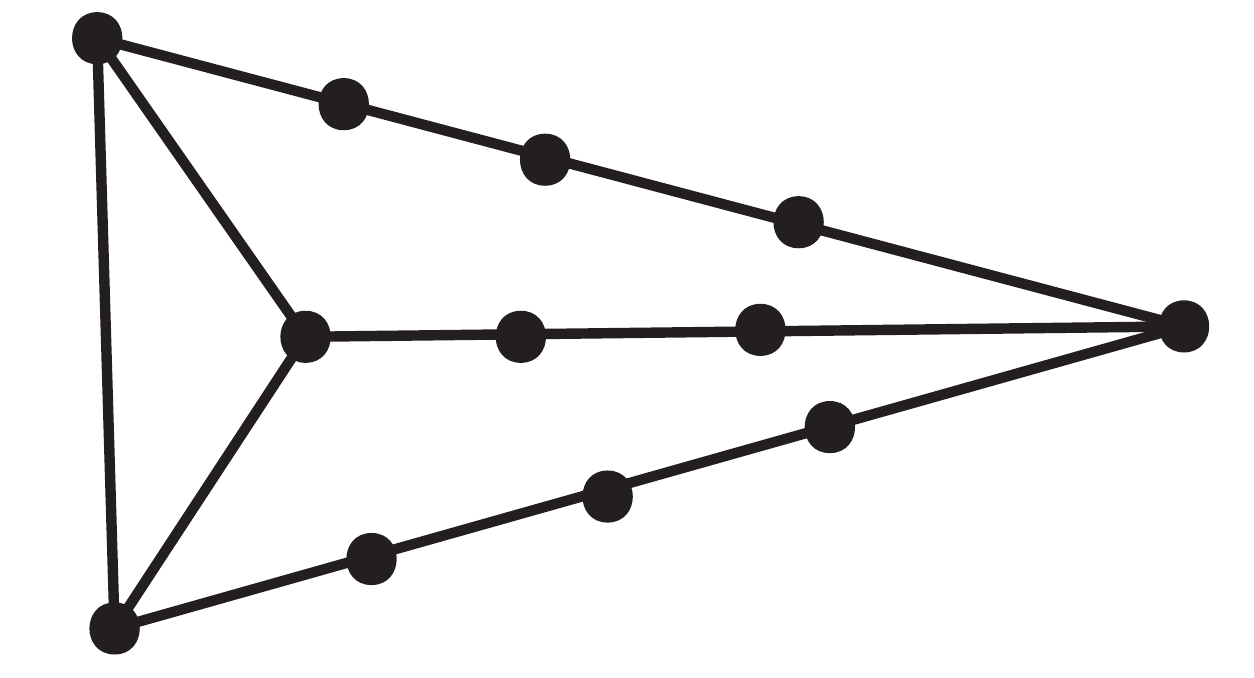}
	\end{center}
	\caption{The graph $SK_4^{3,4,4}$.}
	\label{SK4picture}
\end{figure}

\begin{lemma}\label{SK4bound} Let $s_1,s_2,s_3\in \mathbb{Z^+}$ and $x$ be a real number with $x\geq 2$. Then,
	$$\pi (SK_4^{s_1,s_2,s_3},\, x+1)\leq \frac{x-1}{x+1}\,x^{s_1+s_2+s_3+1} \left(1+\frac{2}{x^2}\right).$$
\end{lemma}
\begin{proof}
	Using \eqref{SK4formula} and Lemma~\ref{thetashift}, calculations show that $\pi ((SK_4)^{s_1,s_2,s_3},\, x+1)$ is equal to
	$$\frac{x(x-1)}{x+1}\left(x^{\sum\limits_{i=1}^3s_i}+(-1)^{s_1+s_2+1}x^{s_3}+(-1)^{s_1+s_3+1}x^{s_2}+(-1)^{s_2+s_3+1}x^{s_1}+2(-1)^{\sum\limits_{i=1}^3s_i}\right).$$
	
	Now, all of $s_1+s_2$, $s_1+s_3$, $s_2+s_3$ cannot be odd at the same time. So at least one of  $s_1+s_2$, $s_1+s_3$, $s_2+s_3$ is even. So this means that at least one of the terms $(-1)^{s_1+s_2+1}x^{s_3}$, $(-1)^{s_1+s_3+1}x^{s_2}$, $(-1)^{s_2+s_3+1}x^{s_1}$ is negative. Therefore it is easy to see that 
	
	$$x^{\sum\limits_{i=1}^3s_i}+(-1)^{s_1+s_2+1}x^{s_3}+(-1)^{s_1+s_3+1}x^{s_2}+(-1)^{s_2+s_3+1}x^{s_1}+2(-1)^{\sum\limits_{i=1}^3s_i}$$
	is at most
	$$ x^{\sum\limits_{i=1}^3s_i} \left(1+\frac{2}{x^2}\right)$$
for every real $x\geq 2$.	Thus the result follows.
\end{proof}

\subsubsection{A subdivision of $K_{3,t}$}

\begin{lemma}\label{k33formula}
Let $\{a,b,c\}$ and $\{v_1,v_2,\dots ,v_t\}$ be the bipartition of the graph $K_{3,t}$. Let $G$ be a subdivision of $K_{3,t}$ such that the edge $av_i$ (resp. $bv_i$ and $cv_i$) of $K_{3,t}$  is replaced with a path of size $a_i$ (resp. $b_i$ and $c_i$) for $i=1,\dots ,t.$ Then $\pi(G,x)$ is equal to
\begin{eqnarray*}
&&\frac{\prod_{i=1}^t\, \pi (\theta_{a_i+1,b_i,c_i},x)}{\left(x(x-1)\right)^{t-1}}+\frac{\prod_{i=1}^t\, \pi (\theta_{a_i,b_i+1,c_i},x)}{\left(x(x-1)\right)^{t-1}}+\frac{\prod_{i=1}^t\, \pi (\theta_{a_i,b_i,c_i+1},x)}{\left(x(x-1)\right)^{t-1}}\\
&&+\frac{\prod_{i=1}^t\, \pi (\theta_{a_i,b_i,c_i},x)}{x^{t-1}}+ \frac{\prod_{i=1}^t \, \pi(SK_4^{a_i,b_i,c_i},x)}{\left(x(x-1)(x-2)\right)^{t-1}}
\end{eqnarray*}
\end{lemma}
\begin{proof}
We apply the addition contraction formula successively. Let $A=G+ab$ and $B=G/ ab$. So, $$\pi(G,x)=\pi(A,x)+\pi(B,x).$$ Let $u$ be the vertex of $B$ which is obtained by contracting $a$ and $b$.  $B_1=B+uc$ and $B_2=B/ uc$. So, $$\pi(B,x)=\pi(B_1,x)+\pi(B_2,x).$$ Let $A_1=A+bc$ and $A_2=A/bc$. So,
$$\pi(A,x)=\pi(A_1,x)+\pi(A_2,x.)$$
Let $A_1^1=A_1+ac$ and $A_1^2=A_1/ ac$. So,
$$\pi(A_1,x)=\pi(A_1^1,x)+\pi(A_1^2,x).$$
Hence, we obtain that
$$\pi(G,x)=\pi(A_1^1,x)+\pi(A_1^2,x)+\pi(A_2,x)+\pi(B_1,x)+\pi(B_2,x).$$
Now we use the Complete Cutset Theorem to find the chromatic polynomials of the graphs $A_1^1$, $A_1^2$, $A_2$, $B_1$, $B_2$.
Observe that $A_1^1$ is the $3$-clique sum of  $(SK_4)^{a_1,b_1,c_1},\dots , (SK_4)^{a_t,b_t,c_t}$. Hence, 
$$\pi(A_1^1,x)=\frac{\prod_{i=1}^t \, \pi(SK_4^{a_i,b_i,c_i},x)}{\left(x(x-1)(x-2)\right)^{t-1}}$$
The graph $A_1^2$ is the $2$-clique sum of $\theta_{a_1,b_1+1,c_1},\, \dots \, ,\theta_{a_t,b_t+1,c_t}$, so
$$\pi(A_1^2,x)=\frac{\prod_{i=1}^t\, \pi (\theta_{a_i,b_i+1,c_i},x)}{\left(x(x-1)\right)^{t-1}}$$
Similarly, $A_2$ is the $2$-clique sum of $\theta_{a_1+1,b_1,c_1},\, \dots \, ,\theta_{a_t+1,b_t,c_t}$; $B_1$ is the $2$-clique sum of $\theta_{a_1,b_1,c_1+1},\, \dots \, ,\theta_{a_t,b_t,c_t+1}$; $B_2$ is the $1$-clique sum of $\theta_{a_1,b_1,c_1},\, \dots \, ,\theta_{a_t,b_t,c_t}$. Therefore,
$$\pi(A_2,x)=\frac{\prod_{i=1}^t\, \pi (\theta_{a_i+1,b_i,c_i},x)}{\left(x(x-1)\right)^{t-1}}$$
$$\pi(B_1,x)=\frac{\prod_{i=1}^t\, \pi (\theta_{a_i,b_i,c_i+1},x)}{\left(x(x-1)\right)^{t-1}}$$
$$\pi(B_2,x)=\frac{\prod_{i=1}^t\, \pi (\theta_{a_i,b_i,c_i},x)}{x^{t-1}}$$
Thus, the result follows.
\end{proof}

\begin{lemma}\label{k33bound}
Let $\{a,b,c\}$ and $\{v_1,v_2,\dots ,v_t\}$ be the bipartition of the graph $K_{3,t}$. Let $G$ be a subdivision of $K_{3,t}$ such that the edge $av_i$ (resp. $bv_i$ and $cv_i$) of $K_{3,t}$  is replaced with a path of size $a_i$ (resp. $b_i$ and $c_i$) for $i=1,\dots ,t.$ Define
$$F(x,t)=3\left(1+\frac{1}{x}+\frac{1}{x^3}+\frac{1}{x^4}\right)^t+\frac{1}{x}\left(1+\frac{2}{x}+\frac{1}{x^2}\right)^t+(x-1)\left(1+\frac{2}{x^2}\right)^t.$$
Then for every real $x\geq 2$,
$$\pi(G,x+1)\leq \frac{x^{n+2t-2}}{(x+1)^{2t-1}}\,F(x,t).$$
\end{lemma}
\begin{proof}
By Lemma~\ref{thetalastbound}, each of $\pi(\theta_{a_i+1,b_i,c_i}, x+1)$, $\pi(\theta_{a_i,b_i+1,c_i}, x+1)$, $\pi(\theta_{a_i,b_i,c_i+1}, x+1)$ is at most $$\frac{x^{a_i+b_i+c_i+1}}{x+1}\left(1+\frac{1}{x}+\frac{1}{x^3}+\frac{1}{x^4}\right)$$ for every real $x\geq \sqrt{2}$. Also, $1+\frac{2}{x}+\frac{1}{x^2}\geq 1+\frac{1}{x}+\frac{1}{x^3}+\frac{1}{x^4} $ holds for all $x\geq 1$. Hence Lemmas~\ref{thetabound} and \ref{thetalastbound} yield
$$\pi(\theta_{a_i,b_i,c_i},x+1)\leq \frac{x^{a_i+b_i+c_i}}{x+1}\left(1+\frac{2}{x}+\frac{1}{x^2}\right).$$
By Lemma~\ref{SK4bound}, we also have
$$\pi(SK_4^{a_i,b_i,c_i},x+1)\leq \frac{x-1}{x+1}x^{a_i+b_i+c_i+1}\left(1+\frac{2}{x^2}\right)$$
for every real $x\geq 2$. Observe that 
$$n+2t-3=\sum\limits_{i=1}^t(a_i+b_i+c_i).$$
Hence, $$x^{n+2t-3}=\prod\limits_{i=1}^t\,x^{a_i+b_i+c_i}.$$ Now by Lemma~\ref{k33formula}, for every real $x\geq 2$, we get
\begin{eqnarray*}
\pi(G,x+1) &\leq & x^{n+2t-2}\left(\frac{3}{(x+1)^{2t-1}}\left(1+\frac{1}{x}+\frac{1}{x^3}+\frac{1}{x^4}\right)^t\right)\\ && + x^{n+2t-2}\left(\frac{1}{x(x+1)^{2t-1}}\left(1+\frac{2}{x}+\frac{1}{x^2}\right)^t\right)\\
&& + x^{n+2t-2}\left(\frac{(x-1)}{(x+1)^{2t-1}}\left(1+\frac{2}{x^2}\right)^t\right)\\
&=& \frac{x^{n+2t-2}}{(x+1)^{2t-1}}\,F(x,t).
\end{eqnarray*}
\end{proof}

\noindent \textbf{Proof of Lemma~\ref{K33son}}.
We shall show that
$$\pi(G,x+1)<(x+1)_{\downarrow 4}\,x^{n-4}$$
holds for every real number $x\geq 2.95$. Take $t=10$ for the rest of the proof. Recall that
$$F(x,t)=3\left(1+\frac{1}{x}+\frac{1}{x^3}+\frac{1}{x^4}\right)^t+\frac{1}{x}\left(1+\frac{2}{x}+\frac{1}{x^2}\right)^t+(x-1)\left(1+\frac{2}{x^2}\right)^t.$$
By Lemma~\ref{k33bound}, it suffices to show that 
$$\frac{x^{n+2t-2}}{(x+1)^{2t-1}}\,F(x,t)< (x+1)_{\downarrow 4}\,x^{n-4}$$ which is equivalent to showing that 
$$x^{2t+2}\, F(x,t)< (x+1)^{2t-1}\, (x+1)_{\downarrow 4}.$$
Calculations show that $x^{2t+2}\,F(x,t)$
is equal to 
$$x^{-2t+1}\left(3x(x^4+x^3+x+1)^t+(x^4+2x^3+x^2)^t+x(x-1)(x^4+2x^2)^t\right).$$
So we shall show that 
$$q(x):=3x(x^4+x^3+x+1)^t+(x^4+2x^3+x^2)^t+x(x-1)(x^4+2x^2)^t$$
is less than
$r(x):=x^{2t-1}(x+1)^{2t-1}(x+1)_{\downarrow 4}$  for all $x\geq 2.95$.
Let 
$$p(x)=r(x)-q(x).$$
Calculations show that for $t=10$, the polynomial $p(x)$ has positive leading coefficient and the largest real root of $p(x)$ is $2.9408\dots$. Thus the result follows.

\subsection{Proof of Lemma~\ref{cactusson}}

\begin{lemma}\label{cactusformula}
	Let $G$ be a cactus graph with $t$ edges and $p$ cycles $C_1,\dots , C_p$ where $|V(C_i)|=n_i$ for $i=1,\dots ,p$. Then
	$$\pi(G,x)=\frac{(x-1)^{t+p}}{x^{p-1}}\, \prod\limits_{i=1}^p ((x-1)^{n_i-1}+(-1)^{n_i})$$
\end{lemma}
\begin{proof}
	By the formula given in equation \eqref{blocks},
	\begin{eqnarray*}
	\pi(G,x) &=& \frac{1}{x^{t+p-1}}\, (x(x-1))^t\,  \prod\limits_{i=1}^p \pi(C_i,x) \\
	&=& \frac{(x-1)^t}{x^{p-1}}\, \prod\limits_{i=1}^p \pi(C_i,x).
	\end{eqnarray*}
	Since $\pi(C_i,x)=(x-1)^{n_i}+(-1)^{n_i}(x-1)$, the latter simplifies to
	$$\frac{(x-1)^{t+p}}{x^{p-1}}\, \prod\limits_{i=1}^p ((x-1)^{n_i-1}+(-1)^{n_i}).$$
	Thus the result follows.
	\end{proof}
	
	\begin{lemma}\label{generalproductbound}
		Let $p, N_1,\dots ,N_p\in \mathbb{Z}^{+}$ be such that $N=\sum\limits_{i=1}^{p}N_i$ and $N_1,\dots ,N_p\geq 3$. Then,
		$$\prod\limits_{i=1}^{p} (x^{N_i}+1) \, \leq x^{N-3p} \, \left(x+\frac{1}{3x^2}\right)^{3p}$$
		for every real $x\geq 1$.
	\end{lemma}
	\begin{proof}
\begin{eqnarray*}
	\prod\limits_{i=1}^{p} (x^{N_i}+1) &\leq & \sum\limits_{i=0}^{p} {p\choose i}\,x^{N-3i}\\
	&=& x^{N-3p} \sum\limits_{i=0}^p {p\choose i} x^{3p-3i}\\
	&\leq & x^{N-3p}\, \sum\limits_{i=0}^{3p} {3p\choose i} x^{3p-i}\left(\frac{1}{3x^2}\right)^i\\
	&= & x^{N-3p}\, \left(x+\frac{1}{3x^2}\right)^{3p}
\end{eqnarray*}	
where the last inequality holds since
$${p\choose i}\leq \frac{1}{3^i}{3p\choose i}$$
for all $i=0,\dots ,p$.
	\end{proof}
	
	\begin{lemma}\label{cactusbound}
		Let $G$ be a cactus graph of order $n$ with $t$ edges and $p$ cycles $C_1,\dots , C_p$ where $|V(C_i)|=n_i$ for  $i=1,\dots ,p$. Then,
		
		$$\pi(G,x+1)\leq \frac{x^{n-8p-1}\, (3x^3+1)^{3p}}{3^{3p}\,(x+1)^{p-1}}$$
		for every real $x\geq 1$.
		\end{lemma}
		
		\begin{proof}
			Assume that exactly $l$ of the cycles $C_1, \dots , C_p$ are equal to $C_3$ where $0\leq l\leq p$. Without loss we may assume $n_1,\dots ,n_l=3$ and $n_{l+1},\dots ,n_p\geq 4$. Also observe that
			$$n=t-p+1+\sum\limits_{i=1}^p n_i$$
			holds. Now,
			\begin{eqnarray}
			\pi(G,x+1) &= & \frac{x^{t+p}}{(x+1)^{p-1}} \, \prod\limits_{i=1}^p(x^{n_i-1}+(-1)^{n_i})  \label{first} \\
			&=&  \frac{x^{t+p}}{(x+1)^{p-1}} \, \prod\limits_{i=1}^l(x^{n_i-1}-1)\, \prod\limits_{i=l+1}^p(x^{n_i-1}+(-1)^{n_i}) \label{second}\\
			&\leq & \frac{x^{t+p}}{(x+1)^{p-1}} \, \prod\limits_{i=1}^lx^{n_i-1}\,\prod\limits_{i=l+1}^p(x^{n_i-1}+1) \label{third} \\
			&\leq & \frac{x^{t+p}}{(x+1)^{p-1}} \ x^{\sum\limits_{i=1}^l(n_i-1)}\,x^{\left(\sum\limits_{i=l+1}^p(n_i-1)\right)-3(p-l)}\,\left(x+\frac{1}{3x^2}\right)^{3(p-l)} \label{new} \\
			&\leq & \frac{x^{t+p}}{(x+1)^{p-1}} \ x^{\sum\limits_{i=1}^l(n_i-1)}\,x^{\left(\sum\limits_{i=l+1}^p(n_i-1)\right)-3p}\,\left(x+\frac{1}{3x^2}\right)^{3p} \label{fourth} \\
			& = & \frac{x^{t+p}}{(x+1)^{p-1}} \, x^{\left(\sum\limits_{i=1}^p (n_i-1)\right)-3p}\, \left(x+\frac{1}{3x^2}\right)^{3p} \label{fifth} \\
			&=& \frac{x^{n-2p-1}}{(x+1)^{p-1}}\,  \left(x+\frac{1}{3x^2}\right)^{3p} \label{sixth} \\
			&=&  \frac{x^{n-8p-1}\, (3x^3+1)^{3p}}{3^{3p}\,(x+1)^{p-1}} \label{son}
			\end{eqnarray}
			
			where \eqref{first} follows by Lemma~\ref{cactusformula}; \eqref{second} holds as $(-1)^{n_i}=-1$ for $i=1,\dots ,l$; \eqref{third} follows because $x^{n_i-1}-1\leq x^{n_i-1}$ and $x^{n_i-1}+(-1)^{n_i}\leq x^{n_i-1}+1$; \eqref{new} holds by
			Lemma~\ref{generalproductbound} (note that if $l=0$ then $\prod\limits_{i=1}^lx^{n_i-1}=1$ and if $l=p$ then $\prod\limits_{i=l+1}^p(x^{n_i-1}+1)=1$); \eqref{fourth} holds since $x^{3l}\left(x+\frac{1}{3x^2}\right)^{-3l}\leq 1$; \eqref{fifth} is clear; \eqref{sixth} holds because $\left(\sum\limits_{i=1}^p (n_i-1)\right)-3p=n-t-1-3p$; \eqref{son} follows by a routine simplification. Therefore we obtain the desired result.
		\end{proof}

		\noindent \textbf{Proof of Lemma~\ref{cactusson}}
		We shall show that $\pi(G,x+1)<(x+1)_{\downarrow 4}\,x^{n-4}$ holds for every real number $x\geq 2.998$.
		Let $p=6$. By Lemma~\ref{cactusbound}, it suffices to show that 
		$$\frac{x^{n-8p-1}\, (3x^3+1)^{3p}}{3^{3p}\,(x+1)^{p-1}}\leq (x+1)_{\downarrow 4}\,x^{n-4}$$
		which is equivalent to showing that 
		$$(3x^3+1)^{3p}\leq 3^{3p}x^{8p-3}(x+1)_{\downarrow 4}(x+1)^{p-1}.$$
		Let 
		$$q(x)=3^{3p}x^{8p-3}(x+1)_{\downarrow 4}(x+1)^{p-1}-(3x^3+1)^{3p}.$$
		Calculations show that the polynomial $q(x)$ has  positive leading coefficient and the largest real root of $q(x)$ is equal to $2.99791\dots$. Hence the result follows.
	\qed

\section{Concluding Remarks}

To prove our main result we reduced the problem to $3$-connected graphs and made use of typical subgraphs of $3$-connected graphs which are large enough. Existence of such typical subgraphs guarantee that the number of $x$-colorings cannot exceed the desired upper bound. Consequently a natural question to ask is what typical subgraphs do $4$-chromatic graphs have and can we make use of such subgraphs to settle the problem? A well known result due to Dirac \cite{dirac} says that every $4$-chromatic graph has a subgraph that is a subdivision of $K_4$. But unfortunately existence of a subdivision of $K_4$ is not helpful. For example, consider $G=SK_4^{3,4,4}$ which is depicted in Figure~\ref{SK4picture}. Then $G$ has $12$ vertices, 
$$\pi(G,x)={x}^{12}-14\,{x}^{11}+90\,{x}^{10}-352\,{x}^{9}+935\,{x}^{8}-\cdots $$
and 
$$(x)_{\downarrow 4}(x-1)^8={x}^{12}-14\,{x}^{11}+87\,{x}^{10}-318\,{x}^{9}+762\,{x}^{8}-\cdots .
$$
Calculations show that for every real $x>2$,
$$\pi(G,x)\nleq (x)_{\downarrow 4}(x-1)^8.$$

Also, Conjecture~\ref{tomesdongconj} (for $k=4$) was proven in \cite{tomescu} for planar graphs. Therefore, by Lemma~\ref{3connreduction}, it suffices to restrict our attention to $3$-connected nonplanar graphs. It is known that every $3$-connected nonplanar graph distinct from $K_5$ contains a subdivision of $K_{3,3}$ (see, for example, \cite{kelmans}). Note that if $G$ is a subdivision of $K_{3,3}$ then the inequality $\pi(G,x)<(x)_{\downarrow 4}(x-1)^{|V(G)|-4}$ does not hold for every $x\geq 4$, however we believe that it holds for $x\geq 7.405.$
\vskip0.4in

\bibliographystyle{elsarticle-num}

\begin{thebibliography}{10}

\bibitem{brownerey}
J. Brown, A. Erey,
New bounds for chromatic polynomials and chromatic roots,
\textit{Discrete Math.} 338(11) (2015) 1938--1946.

\bibitem{brownsokal}
J.I. Brown, C. Hickman, A.D. Sokal, D.G. Wagner,
On the Chromatic Roots of Generalized Theta Graphs,
\textit{J. Combin. Theory Ser. B} 83 (2001) 272--297.

\bibitem{dirac}
G.A. Dirac,
The structure of $k$-chromatic graphs and some remarks on critical graphs,
\textit{J. Lond. Math. Soc.} 27 (1952) 269--271.

\bibitem{dongbook}
Dong, F.M., Koh, K.M. and Teo, K.L., \textit{Chromatic Polynomials And
Chromaticity Of Graphs}, World Scientific, London, (2005).

\bibitem{engbers}
J. Engbers,
Maximizing H-colorings of connected graphs with fixed minimum degree,
\textit{J. Graph Theory}, to appear.

\bibitem{engbersgalvin}
J. Engbers, D. Galvin,
Extremal $H$-colorings of trees and $2$-connected graphs,
\textit{J. Combin. Theory Ser. B}, to appear.



\bibitem{ereyjoc}
A. Erey,
On the maximum number of colorings of a graph,
submitted.

\bibitem{loh}
P.S. Loh, O. Pikhurko, B. Sudakov,
Maximizing the number of $q$-colorings,
\textit{Proc. London Math. Soc.} 101(3) (2010) 655–696.


\bibitem{kelmans}
A.K. Kelmans,
A strengthening of the Kuratowski planarity criterion for $3$-connected graphs,
\textit{Discrete Math.} 51 (1984) 215--220.

\bibitem{ma}
J. Ma, H. Naves,
Maximizing proper colorings on graphs,
\textit{J. Combin. Theory Ser. B} 115 (2015) 236--275.

\bibitem{oxley}
B. Oporowski, J. Oxley, R. Thomas,
Typical Subgraphs of $3$- and $4$-Connected Graphs,
\textit{J. Combin. Theory Ser. B} 57(2) (1993) 239--257.

\bibitem{tomescufrench}
I. Tomescu,
Le nombre des graphes connexes $k$-chromatiques minimaux aux sommets \'{e}tiquet\'{e}s, \textit{C.~R. Acad. Sci. Paris} 273 (1971) 1124--1126.

%


\bibitem{tomescu}
I. Tomescu,
Maximal Chromatic Polynomials of Connected Planar Graphs,
\textit{J. Graph Theory} 14 (1990) 101--110.

\bibitem{tofts}
S.N. Tofts,
An Extremal Property of Tur{\'a}n Graphs, II,
\textit{J. Graph Theory} 75(3) (2014) 275--283.



\bibitem{westbook}
D.B. West,
Introduction to Graph Theory, second ed., Prentice Hall, New York, 2001.


\bibitem{zhao}
Y. Zhao,
The Bipartite Swapping Trick on Graph Homomorphisms, 
\textit{SIAM J. Discrete Math.} 25 (2011) 660--680.

%
%
%
%
%
%
%
%

%
%
%
%

%

\end{thebibliography}

\end{document}